\documentclass[11pt, reqno]{amsart}

\usepackage{amssymb}
\usepackage{amsfonts}
\usepackage{amsmath}
\usepackage{amsxtra}
\usepackage{amsthm}
\usepackage{mathrsfs}
\usepackage{empheq}
\usepackage{enumerate}

\newtheorem{theorem}{Theorem}[section]
\newtheorem{corollary}[theorem]{Corollary}
\newtheorem{lemma}[theorem]{Lemma}
\newtheorem{proposition}[theorem]{Proposition}

\theoremstyle{remark}
\newtheorem*{remark}{Remark}

\theoremstyle{definition}
\newtheorem*{definition}{Definition}

\numberwithin{equation}{section}

\def\rit{\mathbb R}
\def\ds{\displaystyle}
\def\eps{\varepsilon}

\newcommand{\paral}{\|}

\def\blfootnote{\xdef\@thefnmark{}\@footnotetext}
\title[$L^2$ Stability Estimates for Shocks]{$L^2$ Stability Estimates for Shock Solutions of Scalar Conservation Laws Using the Relative Entropy Method} 

\author[N. Leger]{Nicholas Leger}
\address{Dept. of Mathematics, 1 University Station C1200, Austin, TX 78712-0257}
\email{nleger@math.utexas.edu}

\begin{document}
\bibliographystyle{plain}

\maketitle
\begin{abstract}
We consider scalar nonviscous conservation laws with strictly convex flux in one spatial dimension, and we investigate the behavior of bounded $L^2$ perturbations of shock wave solutions to the Riemann problem using the relative entropy method. We show that up to a time-dependent translation of the shock, the $L^2$ norm of a perturbed solution relative to the shock wave is bounded above by the $L^2$ norm of the initial perturbation.
\end{abstract}

\medskip
\noindent {Keywords:} conservation laws; relative entropy method; riemann problem; $L^2$ stability; shocks.
\vskip 0.4cm

\medskip
\noindent
{AMS Subject Classification:}
 35L45, 
 35L65, 
 35L67. 
\vskip 0.4cm

\section{Introduction}\label{intro}
For scalar nonviscous conservation laws with general flux, it is well-known from the theory of Kru\v{z}kov \cite{Kruzkov} that the solution operator for the initial value problem
\begin{equation}\label{introcl}\left\{\begin{array}{l}
\partial_t U + \partial_x A(U) = 0;\\[0.2 cm]
U(x, 0) = U^0(x),
\end{array}\right.
\end{equation}
forms an $L^1$-contraction semigroup. As a result, the measure in $L^1(\rit)$ of the difference of any pair of entropy solutions is non-increasing over time. In particular, if $\phi(x-\sigma t)$ is a shock wave and $U^0-\phi \in L^1(\rit)$, then
\begin{align}\label{kruzkovestimate}
\paral U(\cdot , t) -\phi(\cdot -\sigma t) \paral_{L^1(\rit)} \leq \paral U^0 -\phi \paral_{L^1(\rit)}
\end{align}
for all $t\geq0$. 
While Kru\v{z}kov's estimate is only valid for scalar equations (Temple \cite{Temple}), global stability estimates for shocks with respect to the $L^1$ metric have also been obtained for hyperbolic systems of conservation laws, at least for sufficiently small perturbations of suitably weak shock waves. One such result is given by Bressan et al. in \cite{Bressan}, where the authors establish $L^1$ stability estimates corresponding to an "almost" contractive semigroup structure.

On the other hand, for systems of conservation laws with a convex extension, it is well-known that the relative entropy method developed by Dafermos \cite{Dafermos3} and DiPerna \cite{DiPerna} provides $L^2$ stability estimates for solutions away from shocks. As illustrated in \cite{Dafermos4}, one gets estimates of the form
\begin{align}
\paral U( \cdot, t) - \overline{U}(\cdot, t) \paral_{L^2([-R,R])} \leq a e^{bt} \paral U^0 - \overline{U}^0 \paral_{L^2([-R-st, R+st])},
\end{align}
where $U$ and $\overline{U}$ are weak and strong solutions, respectively, and where $a$, $b$, and $s$ are constants depending on the initial data $U_0$ and $\overline{U}_0$.  However, simple examples show that this kind of result cannot hold when $\overline{U}$ is a shock wave or more generally when $\overline{U}$ is only a weak solution of the conservation law. Nonetheless, as we will prove, one can expect similar estimates to hold up to a suitable translation of the shock.

Consider the initial value problem (\ref{introcl}) for a scalar conservation law in one space dimension. Our goal is to prove the the following global $L^2$ stability estimate for shocks.
\begin{theorem}\label{strong}
Let $U^0 \in L^{\infty}(\rit)$ and assume $U^0 -\phi \in L^2(\rit)$ where
\begin{align}\label{riemanndata}
\phi(x) = 
\begin{cases}
C_L, &\text{if $x<0$;}\\
C_R, &\text{if $x>0$,}\\
\end{cases}
\end{align}
with $C_L > C_R$. Further, assume $U$ is the unique entropy solution of (\ref{introcl}), for a smooth flux function $A: \rit \to \rit$ verifying $A^{\prime \prime} > 0$. Then there exists a Lipschitz continuous function $\overline{x}: [0,\infty) \to \rit$ and a constant $\lambda(\paral U^0 \paral_{L^\infty};\phi; A)$ $>0 $ such that 
\begin{align}\label{mainestimate}
\paral U(\cdot , t) -\phi(\cdot -\sigma t - \overline{x}(t)) \paral_{L^2(\rit)} \leq \paral U^0 -\phi \paral_{L^2(\rit)}
\end{align}
and 
\begin{align}\label{introshiftbound}
\vert \overline{x}(t) \vert \leq \lambda \paral U^0 -\phi \paral_{L^2(\rit)} \sqrt{t}
\end{align}
for all $t \geq 0$, where $\sigma$ is given by the relation $\sigma (C_L - C_R) = A(C_L)-A(C_R)$.
\end{theorem}
\noindent

In the same spirit as \cite{Bressan}, our result can be characterized as an almost contractive variation of Kru\v{z}kov's estimate (\ref{kruzkovestimate}). In fact, we will prove a slightly more general result (Theorem \ref{strong2}) in Section \ref{maintheorem} which takes into account all strictly convex entropies associated with (\ref{introcl}). However, while Theorem \ref{strong2} is interesting in its own right, it is important to keep in mind that the estimates gained from the relative entropy method are purely of type $L^2$, regardless of the specific convex entropy used.

In a related result, Chen et al. \cite{Chen1} have used the relative entropy method to obtain stability estimates for shocks in the context of gas dynamics. Specifically, the authors establish the time-asymptotic stability of Riemann solutions with arbitrarily large oscillation for the $3 \times 3$ system of Euler equations in one space dimension. The present work is another attempt at developing a stability theory for shocks using relative entropy techniques, beginning with a treatment of scalar equations, as proposed by Vasseur in \cite{Handbook}. For further applications of the relative entropy method in the context of fluid dynamics and kinetic theory, we refer the reader to the papers \cite{Vasseur, Vasseur3, Brenier1, Brenier2, Otto, Vasseur2, Yau}.

Let us clarify a few things regarding estimates (\ref{mainestimate}) and (\ref{introshiftbound}). First, we are interested in controlling the relative entropy globally in time. Therefore, unlike the formulation of many large-time stability results (cf. \cite{Serre, Oleinik, Jones, Liu1, Liu2}), we are forced to take a $\mathit{time}$-$\mathit{dependent}$ translation of the shock. Also, we should point out that (\ref{introshiftbound}) is only intended to show that the shift $\overline{x}(t)$ has a bound proportional to the size of the initial perturbation. The estimate captures neither the Lipschitz continuity of $\overline{x}$ at $t=0$, nor the expected large-time behavior of $\overline{x}$. Indeed, based on our techniques, it is reasonable to expect the time-asymptotic convergence of $\overline{x}(t)$ to the shift, $x_0$, considered in the large-time stability analysis of Il'in and Oleinik \cite{Oleinik}, at least when the initial perturbation is in $L^1(\rit)$. 

The proof of Theorem \ref{strong} relies heavily on the work of Dafermos; in particular, on the theory of generalized characteristics and contingent equations considered in \cite{Dafermos1}. Roughly, the idea is to find curves $x_L(t)$ and $x_R(t)$, initially positioned at the origin, which preserve the quantities
\begin{align*}
\displaystyle \mathcal{E}_L(t) &= \int_{-\infty}^{x_L(t)} \left( U(x,t) - C_L \right)^2 \,dx, \text{ and}\\[0.2 cm]
\displaystyle \mathcal{E}_R(t) &= \displaystyle \int_{x_R(t)}^{\infty} \left( U(x,t) - C_R \right)^2 \,dx.
\end{align*}
Quite surprisingly, this condition leads to the constraint $x_R(t) \leq x_L(t)$ for all $t\geq0$, so that (\ref{mainestimate}) holds for all functions $\overline{x}(t)$ satisfying  $x_R(t) - \sigma t \leq  \overline{x}(t) \leq x_L(t) - \sigma t$. 
The bound on $\overline{x}$ then follows easily from the local $L^1$ stability of entropy solutions.

The paper is organized as follows. In Section \ref{prelim}, we introduce the entropy and relative entropy inequalities associated with the conservation law (\ref{introcl}), and establish some preliminary estimates related to those quantities. In Section \ref{maintheorem}, we describe our method in detail and present the proof of the main theorem. Finally, we include in the appendix an existence result for differential inclusions arising in the context of conservation laws.

\section{Relative Entropy Estimates}\label{prelim}
Consider the scalar conservation law
\begin{align}\label{cl2}
\partial_t U + \partial_x {A(U)} = 0,
\end{align}
with smooth flux $A: \rit \to \rit$. We say that $U \in L^{\infty}( \rit \times (0,\infty))$ is an entropy solution of (\ref{cl2})
if $U$ satisfies
\begin{align*}
\partial_t {\eta (U)} + \partial_x {G(U)} \leq 0, 
\end{align*}
in the sense of measures, for all entropy/entropy-flux pairs $(\eta, G) \in [C^{\infty}(\rit)]^2$ verifying
$\eta^{\prime \prime} \geq 0$ and $G^{\prime} = \eta^{\prime} A^{\prime}$.
If $\overline{U}$ solves (\ref{cl2}) in the classical sense, and we consider, associated to each convex $\eta$, the relative entropy function
\begin{align*}
\eta(U \mid \overline{U}) &= \eta(U) - \eta(\overline{U}) - \eta^{\prime}(\overline{U}) \cdot (U - \overline{U}),
\end{align*}
then an entropy solution $U$ verifies additionally the inequality
\begin{align}\label{relative2}
\partial_t \left[ \eta(U \mid \overline{U}) \right] &+ \partial_x \left[ G(U) - G(\overline{U}) - \eta^{\prime}(\overline{U}) \cdot (A(U) - A(\overline{U})) \right] \\
&\leq -\partial_t \left[ \eta^{\prime}(\overline{U}) \right] \cdot (U - \overline{U}) -\partial_x \left[ \eta^{\prime}(\overline{U}) \right] \cdot (A(U) - A(\overline{U})), \notag
\end{align}
for the same entropy/entropy-flux pairs $(\eta, G)$.

The idea of the relative entropy method is to use (\ref{relative2}) to estimate the quantity $\int \eta(U \mid \overline{U}) \,dx$ in time. When $\eta(U) = U^2$, this corresponds directly to estimates in the $L^2$ metric. In this paper, we will only consider the case $\overline{U}(x,t) = C$, for constant states $C \in \rit$, so that (\ref{relative2}) reduces to
\begin{align}\label{relative3}
\partial_t \eta(U \mid C) + \partial_x  F(U,C) \leq 0,
\end{align}
where
\begin{align}\label{entropyflux}
F(U,C) &= G(U) - G(C) - \eta^{\prime}(C) \cdot (A(U) - A(C)) \notag \\[0.2 cm]
&=\displaystyle \int_C^U (\eta^{\prime}(w) - \eta^{\prime}(C))A^{\prime}(w) \,dw ,
\end{align}
represents the flux of relative entropy. 

The goal of this section is to establish some preliminary estimates related to relative entropy inequality (\ref{relative3}). Before presenting those results, let us remind the reader of the following facts.
\begin{remark}\label{remark1}
Throughout the paper, we assume that $U$ is an entropy solution of (\ref{cl2}) with initial data $U^0 \in L^{\infty}(\rit)$. We assume also that the flux function, $A:\rit \to \rit$, is strictly convex and smooth, so that $A^{\prime \prime} \circ U \geq \alpha > 0$. With these assumptions, Oleinik's estimate implies that $U(\cdot, t) \in BV_{loc}(\rit)$ for all $t>0$. Therefore, the one-sided limits $U(x-, t)$ and $U(x+,t)$ exist and verify $U(x-, t) \geq U(x+,t)$ for all $x \in \rit$ and for all $t>0$.  Furthermore, the trace theorem of Vasseur \cite{Vasseur4} shows that $U(\cdot, t)$ is countinuous with values in $L^1_{loc}(\rit)$ up to $t=0$.
\end{remark}
We begin with the following estimate, which is an adaptation of a key lemma in \cite{Dafermos1} [Lemma 3.2] as it applies to (\ref{relative3}).

\begin{lemma}\label{innerentropy}
Assume $\eta: \rit \to \rit$ is smooth and convex, and let $x_L: [t_0, T] \to \rit$ and $x_R: [t_0, T] \to \rit$ be Lipschitz continuous functions such that for all $t \in [t_0, T]$,  $x_R(t) - x_L(t) \geq \delta$ for a fixed $\delta>0$. Then for $C \in \rit$ and for all $a$ and $b$ with $t_0 \leq a < b \leq T$,
\begin{align}\label{entropy2b}
\displaystyle \int_{x_L(b)}^{x_R(b)} \eta(U(x,b) & \mid C)  \,dx - \displaystyle \int_{x_L(a)}^{x_R(a)} \eta(U(x,a) \mid C) \,dx\\
&\leq \displaystyle \int_a^b F(U(x_L(t)+,t),C) - \dot{x_L}(t) \eta(U(x_L(t)+,t) \mid C) \,dt \notag\\
&- \displaystyle \int_a^b F(U(x_R(t)-,t),C) - \dot{x_R}(t) \eta(U(x_R(t)-,t) \mid C) \,dt, \notag
\end{align}
for any entropy solution $U \in L^{\infty}( \rit \times (0,\infty))$ of (\ref{cl2}).
\end{lemma}

\begin{proof}
For $\eps < \frac{\delta}{2}$, let $\psi_{\eps}(\cdot,t)$ be an "inner" approximation to the characteristic function on $[x_L(t), x_R(t)]$ (instead of the "outer" approximation taken in \cite{Dafermos1}). Specifically, let
\begin{align*}
\psi_{\eps}(x,t) = 
\begin{cases}
0, &\text{if $x<x_L(t)$;}\\
\frac{1}{\eps}(x-x_L(t)), &\text{if $x_L(t) \leq x < x_L(t) + \eps$;}\\
1, &\text{if $x_L(t) + \eps \leq x \leq x_R(t) - \eps$;}\\
-\frac{1}{\eps}(x-x_R(t)), &\text{if $x_R(t) - \eps < x \leq x_R(t)$;}\\
0, &\text{if $x_R(t) < x$.}\\
\end{cases}
\end{align*}

\noindent
Also, as in \cite{Dafermos1} let
\begin{align*}
\chi_{\eps}(t) = 
\begin{cases}
0, &\text{if $t<a$;}\\
\frac{1}{\eps}(t-a), &\text{if $a \leq t < a + \eps$;}\\
1, &\text{if $a + \eps \leq t \leq b$;}\\
-\frac{1}{\eps}(t-(b+\eps)), &\text{if $b < t \leq b + \eps$;}\\
0, &\text{if $b + \eps < t$.}\\
\end{cases}
\end{align*}
\noindent
Applying the non-negative test function $\varphi_{\eps}(x,t) =  \psi_{\eps}(x,t) \chi_{\eps}(t)$ to (\ref{relative3}) and taking $\eps \to 0$
yields (\ref{entropy2b}). 
\end{proof}

Now consider the Riemann data $\phi$ given by (\ref{riemanndata}). Continuing in the spirit of \cite{Dafermos1} we have the following estimate.
\begin{lemma}\label{outer}
Let $\phi$ be defined by (\ref{riemanndata}) and suppose $U^0 \in L^{\infty}(\rit)$ satisfies $\paral U^0 -\phi \paral_{L^2(\rit)} < +\infty$. Also, assume $\eta: \rit \to \rit$ is smooth and convex, and let $x_L: [t_0, T] \to \rit$ and $x_R: [t_0, T] \to \rit$ be Lipschitz continuous functions. Then for all $a$ and $b$ with $t_0 \leq a < b \leq T$,
\begin{align}\label{entropy2}
&\left\{\displaystyle \int_{-\infty}^{x_L(b)} \eta(U(x,b) \mid C_L) \,dx + \displaystyle \int_{x_R(b)}^{\infty} \eta(U(x,b) \mid C_R) \,dx \right\} \notag\\
&-\left\{ \displaystyle \int_{-\infty}^{x_L(a)} \eta(U(x,a) \mid C_L) \,dx  + \displaystyle \int_{x_R(a)}^{\infty} \eta(U(x,a) \mid C_R) \,dx \right\} \\
&\leq - \displaystyle \int_a^b F(U(x_L(t)-,t),C_L) - \dot{x_L}(t) \eta(U(x_L(t)-,t) \mid C_L) \,dt \notag\\
&+\displaystyle \int_a^b F(U(x_R(t)+,t),C_R) - \dot{x_R}(t) \eta(U(x_R(t)+,t) \mid C_R) \,dt, \notag
\end{align}
where $U$ is the unique entropy solution of (\ref{cl2}) with initial data $U^0$.
\end{lemma}

\begin{proof}
Choose R sufficiently large so that $-R < x_L(t) - \eps$ and consider instead the test function 
\begin{align*}
\psi_{\eps,R}(x,t) = 
\begin{cases}
0, &\text{if $x<-2R$;}\\
\frac{1}{R}(x+2R), &\text{if $-2R \leq x < -R$;}\\
1, &\text{if $-R \leq x \leq x_L(t) - \eps$;}\\
-\frac{1}{\eps}(x-x_L(t)), &\text{if $x_L(t) - \eps < x \leq x_L(t)$;}\\
0, &\text{if $x_L(t) < x$.}\\
\end{cases}
\end{align*}
\noindent
Testing (\ref{relative3}) with $\varphi_{\eps,R}(x,t) =  \psi_{\eps,R}(x,t) \chi_{\eps}(t)$ and taking $\eps \to 0$ and $R \to \infty$ we get 
\begin{align*}
\displaystyle \int_{-\infty}^{x_L(b)}  \eta(U(x,b) & \mid C_L) \,dx - \displaystyle \int_{-\infty}^{x_L(a)} \eta(U(x,a) \mid C_L) \,dx.\\
&\leq  \displaystyle \int_a^b \left\{\displaystyle \lim_{R \to \infty} \frac{1}{R} \int_{-2R}^{-R} F(U(x,t),C_L) \,dx \right\} \,dt\\
&- \displaystyle \int_a^b F(U(x_L(t)-,t),C_L) - \dot{x_L}(t) \eta(U(x_L(t)-,t) \mid C_L) \,dt.\\
\end{align*}
\noindent
Since $\paral U^0 -\phi \paral_{L^2(\rit)} < +\infty$ and 
\[\vert F(U(x,t),C_L) \vert \leq \frac{L_{\eta}}{2} \left[ \sup_{\vert w \vert \leq \paral U^0 \paral_{L^{\infty}} + \vert C_L \vert} \vert A^{\prime} (w) \vert \right] (U(x,t) - C_L)^2 \]
on account of (\ref{entropyflux}), the first term on the right-hand side above vanishes. The relative entropy on the right is controlled by a similar argument. 
\end{proof}

\noindent
The following corollary is immediate.

\begin{corollary}\label{outer2}
Assume that the hypotheses of Lemma \ref{outer} are satisfied, and assume additionally that
\begin{enumerate}[(i)]
\item $\dot{x_L}(t) \leq f(U(x_L(t)-,t),C_L)$ for almost every $t \in [t_0, T]$, and
\item $\dot{x_R}(t) \geq f(U(x_R(t)+,t),C_R)$ for almost every $t \in [t_0, T]$,
\end{enumerate} 
where $f(U,C) = \frac{F(U,C)}{\eta(U \mid C)}$. Then, for all $a$ and $b$ with $t_0 \leq a < b \leq T$,
\begin{align*}
&\left\{\displaystyle \int_{-\infty}^{x_L(b)} \eta(U(x,b) \mid C_L) \,dx + \displaystyle \int_{x_R(b)}^{\infty} \eta(U(x,b) \mid C_R) \,dx \right\} \notag\\
&-\left\{ \displaystyle \int_{-\infty}^{x_L(a)} \eta(U(x,a) \mid C_L) \,dx  + \displaystyle \int_{x_R(a)}^{\infty} \eta(U(x,a) \mid C_R) \,dx \right\} \leq 0.
\end{align*}
\end{corollary}

\subsection{The Normalized Relative Entropy Flux}\label{nref}
Consider the function $f: \rit \times \rit \to \rit$, which we call the normalized relative entropy flux, defined by
\begin{align}\label{entropyflow}
f(U,C) = \frac{F(U,C)}{\eta(U \mid C)} = \displaystyle \int_C^U \varphi (w,U,C) A^{\prime}(w) \,dw,
\end{align}
where
\begin{align*}
\varphi (w,U,C) = \displaystyle \frac{\eta^{\prime}(w) - \eta^{\prime}(C)}{\eta(U \mid C)}.
\end{align*}
The estimates obtained in Section \ref{maintheorem} are based on the fact that for strictly convex functions $A$ and $\eta$, the function defined by (\ref{entropyflow}) is Lipschitz and increasing in the variables $U$ and $C$, respectively. In order to prove this, let us first compute the first order partial derivatives of $f$. Assume $U \ne C$. Then,
\begin{align}\label{partialu}
\frac{\partial f}{\partial U} (U,C) &= \varphi (U,U,C) A^{\prime}(U) + \displaystyle \int_C^U \frac{\partial \varphi}{\partial U}(w,U,C) A^{\prime}(w) \,dw \notag \\
& = \varphi(U,U,C) \left[ A^{\prime}(U) - f(U,C)\right]
\end{align}
and
\begin{align}\label{partialc}
\frac{\partial f}{\partial C} (U,C) &= \displaystyle \int_C^U \frac{\partial \varphi}{\partial C}(w,U,C) A^{\prime}(w) \,dw \notag\\
& = \displaystyle \frac{-\eta^{\prime \prime}(C) \cdot (U-C)}{\eta(U \mid C)}\left[ \frac{A(U) - A(C)}{U-C} - f(U,C) \right].
\end{align}

Using these formulas and taking suitable Taylor approximations, one can show that $f$ and its gradient have a continuous extension to the line $U=C$. In particular, since
\begin{align}\label{massone}
\displaystyle \int_C^U \varphi(w,U,C) \,dw = 1,
\end{align}
for all $U \ne C$, we define $f(C,C) = A^{\prime}(C)$ for all $C \in \rit$.

As suggested above, we would like to show that on bounded subsets of $\rit^2$, (\ref{partialu}) and (\ref{partialc}) are bounded by positive constants from above and below, respectively. The latter estimate is delicate and relies on the following lemma.
\begin{lemma}\label{delicate}
Let $g:[0,1] \to \rit$ and $h:[0,1] \to \rit$ be continuously differentiable functions such that $g^{\prime}(s) \geq \eps_g > 0$ and  $h^{\prime}(s) \geq \eps_h > 0$ for all $s \in (0,1)$. Further suppose $\int_0^1 g(s) \,ds = 1$. Then,
\begin{align*}
\int_0^1 g(s) h(s) \,ds - \int_0^1 h(s) \,ds \geq \frac{\eps_g \eps_h}{12}.
\end{align*}
\end{lemma}

\begin{proof}
Since $g$ is increasing and $\int_0^1 g(s) \,ds =1$, there exists a unique $s_0 \in (0,1)$ such that $g(s_0) = 1$. Therefore,
\begin{align*}
\int_0^1 g(s) h(s) \,ds - \int_0^1 h(s) \,ds &= \int_0^1 (g(s)-g(s_0))h(s) \,ds\\
&= \int_0^1 (g(s)-g(s_0))(h(s)-h(s_0)) \,ds\\
&= \int_0^1 \left[ \int_{s_0}^s g^{\prime}(\tau) \,d \tau \right] \left[ \int_{s_0}^s h^{\prime}(\tau) \,d \tau \right] \,ds\\
&\geq \eps_g \eps_h \int_0^1 (s-s_0)^2 \,ds \geq  \frac{\eps_g \eps_h}{12},
\end{align*}
where we used the fact that the last term is minimized at $s_0 = \frac{1}{2}$.
\end{proof}

We can now prove our original claim.
\begin{lemma}\label{lipz}
Assume $A:\rit \to \rit$ and $\eta:\rit \to \rit$ are smooth, strictly convex functions, and let $f: \rit \times \rit \to \rit$ be defined by (\ref{entropyflow}). Then, for bounded sets $\Omega \subset \rit \times \rit$, we have
\begin{enumerate}[(i)]
\item $0 \leq \displaystyle \frac{\partial f}{\partial U} \bigg \vert_{\Omega} \leq L_{\Omega}$, and
\item $0 < \eps_{\Omega} \leq \displaystyle \frac{\partial f}{\partial C} \bigg \vert_{\Omega}$,
\end{enumerate}
where $\eps_\Omega$ and $L_\Omega$ are constants depending on $\Omega$.
\end{lemma}

\begin{proof}
First, choose $R$ sufficiently large so that $\vert (U, C) \vert \leq R$ for all $(U,C) \in \Omega$, and let $\eps_A$, $\eps_{\eta}$, $L_A$, and $L_{\eta}$ be positive constants such that
\begin{align}\label{abound}
0 < \eps_A \leq A^{\prime \prime} (\xi) \leq L_A
\end{align}
and 
\begin{align}\label{etabound}
0 < \eps_{\eta} \leq {\eta}^{\prime \prime} (\xi) \leq L_{\eta}
\end{align}
for all $\vert \xi \vert \leq R$. 
Taking into account (\ref{massone}) and the fact that $A^{\prime}$ and ${\eta}^{\prime}$ are increasing, the quantity computed in (\ref{partialu}), for $U \ne C$, can be controlled in the following way. 
\begin{align*}
0 \leq \frac{\partial f}{\partial U}(U,C) & \leq \varphi (U,U,C) \left[ A^{\prime}(U) - A^{\prime}(C)\right] \\
& = \frac{\int_C^U {\eta}^{\prime \prime} (\xi) \,d \xi}{\eta(U \mid C)} \left[ \int_C^U A^{\prime \prime} (\xi) \,d \xi \right]\\[0.2 cm]
& \leq \frac{L_{\eta} L_A (U-C)^2}{\eps_{\eta} (U-C)^2}  = \frac{L_{\eta} L_A}{\eps_{\eta}}.
\end{align*}
Since the bound is independent of $ \vert U-C \vert$, the estimate holds also for $U=C$. This proves part (i) of the lemma.

Next, assume $U \ne C$ and observe that, after a change of variables, equation (\ref{partialc}) can be rewritten in the form
\begin{align}\label{partialc2}
\frac{\partial f}{\partial C} (U,C)  &= \displaystyle \frac{-\eta^{\prime \prime}(C) \cdot \vert U-C \vert}{\eta(U \mid C)} \left[\int_0^1 h(s) \,ds - \int_0^1 g(s)h(s) \,ds \right],
\end{align}
where
\begin{align*}
h(s) = sgn(U-C) A^{\prime}(C+s(U-C))
\end{align*}
and
\begin{align*}
g(s) =(U-C) \varphi(C+s(U-C), U, C).
\end{align*}
Let us check that $g$ and $h$ satisfy the hypotheses of Lemma \ref{delicate}. First, according to (\ref{massone}), we have $\int_0^1 g(s) \,ds = 1$. Moreover, we compute
\begin{align*}
g^{\prime}(s) =\frac{(U-C)^2 \eta^{\prime \prime}(C+s(U-C))}{\eta(U \mid C)} \geq \frac{{\eps}_{\eta}(U-C)^2}{\frac{L_{\eta}}{2} (U-C)^2} = \frac{2{\eps}_{\eta}}{L_{\eta}} > 0.
\end{align*}
and
\begin{align*}
h^{\prime}(s) = \vert U-C \vert A^{\prime \prime}(C+s(U-C)) \geq \vert U-C \vert {\eps}_A > 0
\end{align*}
Therefore, Lemma \ref{delicate} applies and we deduce from (\ref{partialc2}) that
\begin{align*}
\frac{\partial f}{\partial C} (U,C)  &\geq \displaystyle \frac{-\eta^{\prime \prime}(C) \cdot \vert U-C \vert}{\eta(U \mid C)} \left[-\frac{ \left(\frac{2{\eps}_{\eta}}{L_{\eta}} \right) \vert U-C \vert {\eps}_A}{12} \right]\\
& \geq \frac{{\eps}_A}{12} \left(\frac{{2\eps}_{\eta}}{L_{\eta}} \right)^2 = \frac{{\eps}_A {\eps_{\eta}}^2}{3{L_{\eta}}^2}  > 0.
\end{align*}
Again, since this bound does not depend on $\vert U-C \vert$, the estimate extends to the line $U=C$. Therefore, we have (ii) and the proof of the lemma is complete. 
\end{proof}

Note that the proof of Lemma \ref{lipz} in the case $\eta(U) = \frac{U^2}{2}$ is much simpler. Indeed, a change of variables in (\ref{entropyflow}) yields
\begin{align}\label{eflux}
f(U,C) &= \displaystyle 2 \int_0^1 sA^{\prime}(C+(U-C)s) \,ds.
\end{align}
Therefore, differentiating (\ref{eflux}) with respect to $U$ and $C$ and using (\ref{abound}) we get
\[0 <  \frac{2}{3} \eps_A  \leq \frac{\partial f}{\partial U} \bigg \vert_\Omega \leq \frac{2}{3} L_A \] 
and
\[0 <  \frac{1}{3} \eps_A  \leq \frac{\partial f}{\partial C} \bigg \vert_\Omega \leq \frac{1}{3} L_A.\]
\section{A Relative Entropy Technique for Shocks}\label{maintheorem}
It is well-known that when $A$ is convex and when the initial data $\phi$, given by (\ref{riemanndata}), is non-increasing, the unique entropy solution of the Riemann problem is the traveling shock wave $\phi(x-\sigma t)$, where the shock speed $\sigma$ is given by the Rankine-Hugoniot relation. The goal of this section is to show that traveling shock waves are stable in the sense of Theorem \ref{strong}. In fact, we will show a slightly more general result.

\begin{theorem}\label{strong2}
Let $U^0 \in L^{\infty}(\rit)$ and assume $U^0 -\phi \in L^2(\rit)$ where $\phi$ is given by (\ref{riemanndata}) with $C_L > C_R$. Further, assume $U$ is the unique entropy solution of (\ref{introcl}), for a smooth flux function $A: \rit \to \rit$ verifying $A^{\prime \prime} > 0$. Then, for any smooth $\eta: \rit \to \rit$ verifying ${\eta}^{\prime \prime} > 0$, there exists a Lipschitz continuous function $\overline{x}: [0,\infty) \to \rit$ and a constant $\lambda(\paral U^0 \paral_{L^\infty};\phi; A; \eta) >0 $ such that 
\begin{align}\label{mainestimate2}
\int_{-\infty}^{\infty} \eta (U(x , t) \mid \phi(x -\sigma t - \overline{x}(t))) \,dx \leq \int_{-\infty}^{\infty} \eta (U^0(x) \mid \phi(x)) \,dx < \infty,
\end{align}
and
\begin{align}
\vert \overline{x}(t) \vert \leq \lambda \paral U^0 -\phi \paral_{L^2(\rit)} \sqrt{t}
\end{align}
for all $t \geq 0$, where $\sigma$ is given by the relation $\sigma (C_L - C_R) = A(C_L)-A(C_R)$.
\end{theorem}
In order to show (\ref{mainestimate2}) for a strictly convex entropy $\eta$, the idea is to construct curves $x_L: [0,\infty) \to \rit$ and $x_R:  [0,\infty) \to \rit$, initialized with the data $x_L(0)=x_R(0)=0$, for which the the total relative entropy
\begin{align}\label{outentropy}
\displaystyle \mathcal{E}(t) = \int_{-\infty}^{x_L(t)} \eta(U(x,t) \mid C_L) \,dx + \displaystyle \int_{x_R(t)}^{\infty} \eta(U(x,t) \mid C_R) \,dx
\end{align}
\noindent
is bounded above by $\mathcal{E}(0)$ for all $t \geq 0$. (Note that when $U^0 -\phi \in L^2(\rit)$ and $U^0 \in L^{\infty}(\rit)$, then $\mathcal{E}(0)$ is finite.) Due to the compressive nature of the solution $U(x,t)$, the construction will produce automatically the constraint $x_R(t) \leq x_L(t)$ for all $t \geq 0$. Therefore, (\ref{mainestimate2}) will follow for any function $\overline{x}(t)$ satisfying $x_R(t)  \leq  \overline{x}(t) + \sigma t \leq x_L(t)$. In particular, this includes the curves $\overline{x}(t) =  x_L(t) - \sigma t$ and $\overline{x}(t) = x_R(t) - \sigma t$.

To control the total entropy, the idea is to exploit (\ref{entropy2}) by choosing $\dot{x_L}(t)$  and $\dot{x_R}(t)$ so that the right hand side vanishes. This makes sense at points of continuity of $U$; however, it turns out we do not have as much freedom at points where $U$ is discontinuous. We borrow the following lemma from \cite{Dafermos1}.

\begin{lemma}\label{rh}
Let $x: [t_0, T] \to \rit$, $0 \leq t_0 < T < \infty$ be a Lipschitz continuous function. Then for almost all $t \in [t_0, T]$,
\begin{align*}\label{rh2}
A(U(x(t)+,t)) - A(U(x(t)-,t)) - \dot{x}(t) \left[ U(x(t)+,t) - U(x(t)-,t) \right] = 0.
\end{align*}
\end{lemma}

The lemma simply asserts that if $x(t)$ moves along a discontinuity of $U$ then its derivative must coincide with the shock speed given by the Rankine-Hugoniot condition.

Next, motivated by the idea of generalized characteristics, we consider a curve $x(t)$ solving in the sense of Filippov (see appendix), the differential inclusion

\begin{equation}\label{odeflpv2}
\dot{x}(t) \in [f(U(x(t)+,t),C), f(U(x(t)-,t),C)],
\end{equation}
\noindent
where $f(U,C) = \frac{F(U,C)}{\eta(U \mid C)}$. In view of Lemma \ref{rh}, we find that (\ref{odeflpv2}) is actually quite restrictive.

\begin{proposition}\label{flpvae}
Let $x: [t_0, T] \to \rit$ be a Filippov solution of (\ref{odeflpv2}) on the interval $[t_0, T]$. Then for almost all $t \in [t_0, T]$,
\begin{align*}
\dot{x}(t) = 
\begin{cases}
f(U(x(t)\pm,t),C), &\text{if $U(x(t)-,t) = U(x(t)+,t)$;} \\[0.2 cm] 
\frac{A(U(x(t)+,t)) - A(U(x(t)-,t))}{U(x(t)+,t) - U(x(t)-,t)}, &\text{if $U(x(t)-,t) > U(x(t)+,t)$.} \\
\end{cases}
\end{align*}
\noindent
\end{proposition}


With these facts in mind, we consider functions $x_L: [0,\infty) \to \rit$ and $x_R: [0,\infty) \to \rit$, with initial values $x_L(0) = x_R(0) = 0$ , solving
\begin{equation}\label{odexlr}\left\{\begin{array}{l}
\dot{x_L}(t) \in [f(U(x_L(t)+,t),C_L), f(U(x_L(t)-,t),C_L)],\\[0.2 cm]
\dot{x_R}(t) \in [f(U(x_R(t)+,t),C_R), f(U(x_R(t)-,t),C_R)],
\end{array}\right.
\end{equation}
\noindent
in the sense of Filippov. Existence (and uniqueness beyond $t=0$) is guaranteed by Proposition \ref{flpv} (see appendix) and Lemma \ref{lipz}. Given (\ref{odexlr}), it follows immediately from Corollary \ref{outer2} that the total entropy $\mathcal{E}(t)$ in (\ref{outentropy}) is bounded above by $\mathcal{E}(0)$. Moreover, since  
\[ \eta(U \mid C) = \int_C^U \int_C^w \eta^{\prime \prime}(\xi) \,d \xi \,dw \]
is non-negative, we easily deduce the following lemma.

\begin{lemma}\label{entropybound}
Assume $x_L: [0, \infty) \to \rit$ and $x_R: [0, \infty) \to \rit$ verify (\ref{odexlr}) in the sense of Filippov. Further, assume $x_R(t) \leq  x_L(t)$ for all $t \geq 0$. Then for any function $x: [0, \infty) \to \rit$ satisfying $x_R(t) \leq  x(t) \leq x_L(t)$ for all $t \geq 0$, we have 
\begin{align*}
\displaystyle \int_{-\infty}^{x(t)} \eta(U(x,t) \mid C_L) \,dx + \displaystyle \int_{x(t)}^{\infty} \eta(U(x,t) \mid C_R) \,dx \leq \mathcal{E}(t) \leq \mathcal{E}(0) ,
\end{align*}
for all $t \geq 0$, where $\mathcal{E}$ is defined by (\ref{outentropy}). 
\end{lemma}

Given Lemma \ref{entropybound}, it remains to show that when (\ref{odexlr}) holds and when $C_L > C_R$, then $x_R(t) \leq x_L(t)$ for all $t \geq 0$. Since $x_L$ and $x_R$ coincide at $t=0$ and they move continuously, the idea is to show that $x_R$ cannot pass $x_L$; that is, we would like to argue that when $x_R(t) =  x_L(t)$, then in some sense $\dot{x_R}(t) \leq \dot{x_L}(t)$. This is not precise, of course, since the derivatives may be only measurable; however, note that when $x_R(t) =  x_L(t) = \tilde{x}$ and $(\tilde{x},t)$ is a point of continuity of $U$, then, formally, the monotonicity of $f$ (Lemma \ref{lipz}) implies
\[\dot{x_L}(t) =  f(U(\tilde{x},t), C_L) >  f(U(\tilde{x} ,t), C_R) = \dot{x_R}(t).\]
This suggests that $x_R$ may only pass $x_L$ at a discontinuity of $U$; however, as we will show, Lemma \ref{lipz} does not allow it. We will argue by contradiction, but first let us prove the following lemmas.

\begin{lemma}\label{inner1}
Assume $x_L: [t_0, T] \to \rit$ and $x_R: [t_0, T] \to \rit$ verify (\ref{odexlr}) in the sense of Filippov on the interval $[t_0, T]$, with $C_L > C_R$. Further, assume that $x_R(t) - x_L(t) \geq \delta >0$ for all $t \in [t_0, T]$. Then for $C \in (C_R, C_L)$ there exists $\lambda >0$ independent of $\delta$, such that for all $a$ and $b$ with $t_0 \leq a < b \leq T$,
\begin{align}\label{entropy2c}
\displaystyle \int_{x_L(b)}^{x_R(b)} \eta(U(x,b) \mid C) \,dx &- \displaystyle \int_{x_L(a)}^{x_R(a)} \eta(U(x,a) \mid C) \,dx \leq -\vert S_{a,b} \vert \lambda \leq 0.
\end{align}
\noindent
where $S_{a,b} = \left\{s \in [a,b]  \mid  \dot{x_R}(s) - \dot{x_L}(s) \geq 0 \right\}$.
\end{lemma}

\begin{proof}
Using Lemma \ref{innerentropy}, for all $a$ and $b$ with $t_0 \leq a < b \leq T$,
\begin{align}\label{entropy2e}
\displaystyle & \int_{x_L(b)}^{x_R(b)} \eta(U(x,b) \mid C) \,dx - \displaystyle \int_{x_L(a)}^{x_R(a)} \eta(U(x,a) \mid C) \,dx \notag\\
&\leq \displaystyle \int_a^b  \eta(U(x_L(t)+,t) \mid C) \left[f(U(x_L(t)+,t),C) - \dot{x_L}(t) \right] \,dt \notag\\
&\qquad - \displaystyle \int_a^b \eta(U(x_R(t)-,t) \mid C) \left[ f(U(x_R(t)-,t),C) - \dot{x_R}(t) \right]  \,dt \\
&\leq \displaystyle \int_a^b  \eta(U(x_L(t)+,t) \mid C) \left[f(U(x_L(t)+,t),C) - f(U(x_L(t)+,t),C_L) \right] \,dt \notag\\
&\qquad- \displaystyle \int_a^b \eta(U(x_R(t)-,t) \vert C) \left[ f(U(x_R(t)-,t),C) - f(U(x_R(t)-,t),C_R) \right]  \,dt, \notag
\end{align}
\noindent
where we used (\ref{odexlr}) to get the last inequality. Now, for $t \in S_{a,b}$
\begin{align*}
f(U(x_L(t)+,t),C_L) \leq \dot{x_L}(t) \leq \dot{x_R}(t) \leq f(U(x_R(t)-,t),C_R).
\end{align*}
\noindent
Therefore, since $U \in L^{\infty}$, we deduce using Lemma \ref{lipz} that 
\begin{align}\label{entropy2f}
f(U(x_R(t)-,t),C_L) &- f(U(x_L(t)+,t),C_L) \notag \\[0.3 cm]
&> f(U(x_R(t)-,t),C_L) - f(U(x_R(t)-,t),C_R) \notag\\[0.1 cm]
& = \ds \int_{C_R}^{C_L} \frac{\partial f}{\partial C}(U(x_R(t)-,t), z) \,dz \notag \\[0.2 cm]
& > \eps (C_L-C_R) > 0,
\end{align}
\noindent
for $t \in S_{a,b}$, where $\eps > 0$ is a lower bound on $\frac{\partial f}{\partial C}$. We deduce further, using (\ref{entropy2f}) and Lemma \ref{lipz}, that
\begin{align*}
U(x_R(t)-,t) - U(x_L(t)+,t) > \frac{\eps}{L} (C_L-C_R) >0,
\end{align*}
\noindent
for $t \in S_{a,b}$, where $L > 0$ is an upper bound on $\frac{\partial f}{\partial U}$. Thus, for $t \in S_{a,b}$, either
\begin{enumerate}[(i)]
\item $\eta(U(x_R(t)-,t) \mid C) \geq \kappa (U(x_R(t)-,t)-C)^2 \geq \kappa \left[\frac{\eps}{2 L} (C_L-C_R)\right]^2$, or
\item $\eta(U(x_L(t)+,t) \mid C) \geq \kappa (U(x_L(t)+,t)-C)^2 \geq \kappa \left[\frac{\eps}{2 L} (C_L-C_R)\right]^2$
\end{enumerate}
\noindent
where $\kappa > 0$ is a lower bound for $\frac{1}{2} \eta^{\prime \prime}(\cdot)$. Returning to (\ref{entropy2e}), we get 
\begin{align}\label{entropy2g}
\displaystyle & \int_{x_L(b)}^{x_R(b)} \eta(U(x,b) \mid C) \,dx - \displaystyle \int_{x_L(a)}^{x_R(a)} \eta(U(x,a) \mid C) \,dx \notag\\
&\leq \displaystyle \int_a^b  \eta(U(x_L(t)+,t) \mid C) \left[ -\ds \int_C^{C_L} \frac{\partial f}{\partial C}(U(x_L(t)+,t), z) \,dz \right] \,dt \notag\\
&\qquad - \displaystyle \int_a^b \eta(U(x_R(t)-,t) \mid C) \left[  \ds \int_{C_R}^{C} \frac{\partial f}{\partial C}(U(x_R(t)-,t), z) \,dz \right]  \,dt \notag\\
&\leq - \eps K \displaystyle \int_a^b \left[ \eta(U(x_L(t)+,t) \mid C) + \eta(U(x_R(t)-,t) \mid C) \right] \,dt \notag \\
&\leq - \eps K  \kappa \left[\frac{\eps}{2 L} (C_L-C_R)\right]^2 \vert S_{a,b} \vert,
\end{align}
\noindent
where $K = \text{min} \left\{C_L-C, C-C_R\right\}$ (and we used the fact that $\eta(U \mid C)$ is non-negative). This proves (\ref{entropy2c}) with $\lambda =  \eps K  \kappa \left[\frac{\eps}{2 L} (C_L-C_R)\right]^2$.
\end{proof}

\begin{lemma}\label{inner2}
There exists $1 < \kappa < \infty$ such that for any $\delta > 0$, if $x_L: [t_0, T] \to \rit$ and $x_R: [t_0, T] \to \rit$ verify
\begin{equation*}\left\{\begin{array}{l}
\dot{x_L}(t) \in [f(U(x_L(t)+,t),C_L), f(U(x_L(t)-,t),C_L)],\\[0.2 cm]
\dot{x_R}(t) \in [f(U(x_R(t)+,t),C_R), f(U(x_R(t)-,t),C_R)],\\[0.2 cm]
x_R(t_0) - x_L(t_0) = \delta,
\end{array}\right.
\end{equation*}
\noindent
in the sense of Filippov, with $C_L > C_R$, and if $x_R(t) - x_L(t) \geq \delta$ for all $t \in [t_0, T]$, then $x_R(t) - x_L(t) \leq \kappa \delta$ for all $t \in [t_0, T]$.
\end{lemma}

\begin{proof}
Given $t \in (t_0, T]$, let $S_{t_0,t} = \left\{s \in [t_0,t] \mid \dot{x_R}(s) - \dot{x_L}(s) \geq 0 \right\}$. Then,
\begin{align}\label{ftoc}
x_R(t) - x_L(t) &= \delta + \displaystyle \int_{t_0}^t \left [\dot{x_R}(s) - \dot{x_L}(s) \right] \,ds \notag\\
& \leq \delta + \displaystyle \int_{S_{t_0,t}} \left [\dot{x_R}(s) - \dot{x_L}(s) \right] \,ds.
\end{align}
\noindent
Since $U \in L^{\infty}$, Lemma \ref{lipz} implies $\vert \dot{x_R}(s) - \dot{x_L}(s) \vert$ is bounded by some constant $M_1$. Furthermore, on account of Lemma \ref{inner1}, for $C \in (C_R,C_L)$ there exists $\lambda >0$ such that
\begin{align*}
0 & \leq \displaystyle \int_{x_L(t)}^{x_R(t)} \eta(U(x,t) \mid C) \,dx \leq \displaystyle \int_{x_L(t_0)}^{x_R(t_0)} \eta(U(x,t_0) \mid C) \,dx -\vert S_{t_0,t} \vert \lambda \notag \\[0.2 cm]
& \leq \displaystyle \int_{x_L(t_0)}^{x_R(t_0)} M_2 \,dx -\vert S_{t_0,t} \vert \lambda =  M_2 \delta -\vert S_{t_0,t} \vert \lambda,
\end{align*}
where again we used $U \in L^\infty$ to bound $\eta(U \mid C)$ by a constant $M_2$. Therefore, $\left\vert S_{t_0,t} \right\vert \leq \delta \left( \frac{M_2}{\lambda} \right)$, and we deduce using (\ref{ftoc}) that $x_R(t) - x_L(t) \leq \kappa \delta$, where $\kappa = 1 + \frac{M_1 M_2}{\lambda}$.
\end{proof}

\noindent
We can now prove our previous claim.
\begin{proposition}\label{nopass}
Assume $x_L: [0, \infty) \to \rit$ and $x_R: [0, \infty) \to \rit$ verify (\ref{odexlr})
in the sense of Filippov with $C_L > C_R$. Then $x_R(t) \leq x_L(t)$ for all $t \geq 0$.
\end{proposition}

\begin{proof}
We argue by contradiction. Let $d(t) = x_R(t) - x_L(t)$ and suppose $d(T) > 0$ for some $T>0$. Then for $ 0 < \delta < d(T)$ define 
\[ d^{-1}(\delta) =  \left\{0 \leq t \leq T  \mid d(t) = \delta \right\} \]
and let $t_\delta = \displaystyle \sup_{t \in d^{-1}(\delta)} t $. Since $d$ is continuous and $d(0) = 0$, $d^{-1}(\delta)$ is nonempty and $t_\delta < T$. Also, we must have $d(t) \geq \delta$ for $t \in \left[ t_\delta, T \right]$, otherwise $t_\delta$ would be larger. Therefore Lemma \ref{inner2} applies with $t_0 = t_\delta$ and we conclude that $d(t) \leq \kappa \delta$ for all $t \in \left[ t_\delta, T \right]$. In particular, $d(T) \leq \kappa \delta$ for $\delta$ arbitrarily small. This is a contradiction.
\end{proof}

\noindent
{\bf Proof of Theorems \ref{strong} and \ref{strong2}.} 
Since $U$ is the unique entropy solution of (\ref{introcl}), the inequality (\ref{relative3}) holds in the sense of measures for any strictly convex entropy $\eta$. Thus, given $x_L$ and $x_R$ defined by (\ref{odexlr}), the estimates in Section \ref{prelim} are valid, and we deduce from Proposition \ref{nopass} that there exists a Lipschitz continuous function $\overline{x}: [0,\infty) \to \rit$ such that $x_R(t) \leq  \overline{x}(t) + \sigma t \leq x_L(t)$ for all $t \geq 0$. Also, since $U^0 -\phi \in L^2(\rit)$, we have 
\begin{align*}
\int_{-\infty}^{\infty} \eta (U^0(x) \mid \phi(x)) \,dx \leq \frac{L_{\eta}}{2} \int_{-\infty}^{\infty} (U^0(x) - \phi(x))^2 \,dx < \infty,
\end{align*}
where $L_{\eta} > 0$ is an upper bound on $\eta^{\prime \prime}$. Therefore, $\mathcal{E}(0)$, defined by (\ref{outentropy}), is finite and Lemma \ref{entropybound} implies
\begin{align*}
\int_{-\infty}^{\infty} \eta (U(x , t) \mid \phi(x -\sigma t - \overline{x}(t))) \,dx \leq \mathcal{E}(t) \leq \mathcal{E}(0) = \int_{-\infty}^{\infty} \eta (U^0(x) \vert \phi(x)) \,dx,
\end{align*}
which completes the proof of estimate (\ref{mainestimate2}). Also, this gives (\ref{mainestimate}), for Theorem \ref{strong}, in the case $\eta(U) = U^2$.

Finally, let us show that the function $\overline{x}$ has a bound proportional to the size of $U^0 - \phi$ in $L^2(\rit)$. Recalling that $x_R(t) \leq  \overline{x}(t) + \sigma t \leq x_L(t)$, we have 
\begin{align*}
\vert \overline{x}(t) + \sigma t \vert \leq \max \left\{ \vert x_R(t) \vert, \vert  x_L(t) \vert \right\} \leq Lt,
\end{align*}
since $x_L$ and $x_R$ are Lipschitz (Proposition \ref{flpv}) and $x_L(0) = x_R(0) = 0$. Note that $L$ depends on 
$\paral U^0 \paral_{L^\infty}$, $\phi$, $\eta$ and $A$, as the velocities $\dot{x_L}$ and $\dot{x_R}$ are given by (\ref{entropyflow}). 

Next, observe that $\phi(x - \sigma t - \overline{x}(t)) - \phi(x - \sigma t)$ has support contained in the interval $\left[ -(L+\vert \sigma \vert) t, (L+\vert \sigma \vert) t \right]$. Therefore,
\begin{align}\label{shiftbound}
(C_L - C_R) \vert \overline{x}(t) \vert &= \paral \phi(\cdot - \sigma t - \overline{x}(t)) - \phi(\cdot - \sigma t) \paral_{L^1(B_{(L+\vert \sigma \vert) t})} \notag \\
&\leq \paral \phi(\cdot - \sigma t - \overline{x}(t)) - U(\cdot, t) \paral_{L^1(B_{(L+\vert \sigma \vert) t})} \notag \\
&+ \paral U(\cdot, t) - \phi(\cdot - \sigma t) \paral_{L^1(B_{(L+\vert \sigma \vert) t})}
\end{align}
Then, by the $L^1$-stability theory of Kru\v{z}kov, the last term is bounded by $\paral U^0 - \phi \paral_{L^1(B_{(M+L+\vert \sigma \vert) t})}$, where $M = \sup \left\{ \vert A^{\prime}(w) \vert; \vert w \vert \leq \paral U^0 \paral_{L^\infty} +\paral \phi \paral_{L^\infty} \right\}$. Therefore, proceeding with the estimate (\ref{shiftbound}), using H\"{o}lder's inequality, we get
\begin{align}\label{shiftbound2}
(C_L - C_R) \vert \overline{x}(t) \vert & \leq \sqrt{2(L+\vert \sigma \vert) t} \cdot \paral \phi(\cdot - \sigma t - \overline{x}(t)) - U(\cdot, t) \paral_{L^2(\rit)} \notag \\
&+ \sqrt{2(M + L+\vert \sigma \vert) t} \cdot \paral U^0 - \phi \paral_{L^2(\rit)}
\end{align}
Finally, since $\frac{\eps_{\eta}}{2}(U-\phi)^2 \leq \eta(U \mid \phi) \leq \frac{L_{\eta}}{2}(U-\phi)^2$, using (\ref{mainestimate2}) we have
\begin{align*}
\paral \phi(\cdot - \sigma t - \overline{x}(t)) - U(\cdot, t) \paral_{L^2(\rit)} \leq \sqrt{\frac{L_\eta}{\eps_\eta}} \cdot \paral U^0 - \phi \paral_{L^2(\rit)},
\end{align*}
which together with (\ref{shiftbound2}) implies
\begin{align*}
\vert \overline{x}(t) \vert & \leq \frac{1}{C_L - C_R}\left[ \sqrt{\frac{2 L_\eta}{\eps_\eta}(L+\vert \sigma \vert)} + \sqrt{2(M + L+\vert \sigma \vert)} \right] \paral U^0 - \phi \paral_{L^2(\rit)} \sqrt{t}.
\end{align*}
This completes the proof. Note that, by construction, $\overline{x}$ is actually Lipschitz even though this estimate does not show it.
\section{Acknowledgements}\label{thanks}
The author would like to thank his advisor, Alexis Vasseur, for suggesting the problem and for his constant encouragement and guidance. The author would also like to thank the (NSF-funded) Research Training Group in Applied and Computational Mathematics at the University of Texas at Austin for fellowship support during the preparation of this paper.
\section{Appendix: Filippov Solutions and Conservation Laws}

In this section we include an existence result for differential equations arising in the context of conservation laws. While this result is an easy application of the theory of Filippov \cite{Filippov}, and by no means original, we have found no explicit statement of this kind in the literature.
\begin{definition}\label{flpvsoln} A solution of (\ref{odeflpv}) in the sense of Filippov on an interval $[t_0,T)$ is an absolutely continuous function $x(t)$ for which (\ref{odeflpv}) holds for almost every $t \in [t_0, T)$.
\end{definition}

\begin{proposition}\label{flpv}
Let $U$ be the unique entropy solution of (\ref{introcl}) with $A^{\prime \prime} > 0$ and $U^0 \in L^{\infty}(\rit)$. Also, let $g: \rit \to \rit$ be continuous and non-decreasing. Then for $(x_0, t_0) \in \rit \times [0,\infty)$ there exists a Lipschitz continuous function $x:[t_0, \infty) \to \rit$, with intial value $x(t_0) = x_0$, solving 
\begin{align}\label{odeflpv}
\dot{x}(t) \in [g(U(x(t)+,t)), g(U(x(t)-,t))],
\end{align} 
in the sense of Filippov (Definition \ref{flpvsoln}).
\noindent
Furthermore, if $t_0>0$ and $g$ is Lipschitz on bounded subsets of $\rit$, then the solution $x$ is unique.
\end{proposition}

\begin{proof}
The existence of solutions follows from \cite{Filippov} [Section 2.7, Theorem 1] provided the set valued function $G(x,t) = [g(U(x+,t)), g(U(x-,t))]$ is upper semicontinuous. Roughly this means that in the limit $(x^{\prime}, t^{\prime}) \to (x, t)$, the sets $G(x^{\prime}, t^{\prime})$ will be contained in $G(x,t)$. In the present setting, upper semi-continuity is an immediate consequence of the continuity of $g$ and the structure of solutions to (\ref{introcl}) detailed in \cite{Dafermos1} (see also Remark \ref{remark1}). Also, since $g \circ U \in L^{\infty}$, the Filippov solutions are Lipschitz and defined for all $t \geq t_0$. 

Now assume additionally that g is Lipschitz on bounded subsets of $\rit$, and let us verify that solutions extend uniquely beyond $t=0$. Applying \cite{Filippov} [Section 2.10, Theorem 1], it suffices to check that for any $T>0$ there exists $\ell \in L^1([t_0, T])$ such that for any almost every $(x,t)$ and $(y,t)$ in $\rit \times [t_0, T]$
\[ (x-y) \cdot (g(U(x,t)) - g(U(y,t))) \leq \ell (t) \vert x-y \vert^2. \]
\noindent
The assumptions on $g$ together with Oleinik's well-known decay estimate easily imply the statement above.
\end{proof}

Existence results of this type have appeared implicitly in the work of Marson and Colombo (see \cite{Colombo} and \cite{Marson}) on ODEs related to traffic modeling. In fact, the uniqueness argument above can be found directly in \cite{Colombo} for the case of concave flux functions. The study of contingent equations in the context of conservation laws can also be found in the papers \cite{Bressan2, Dafermos1, Dafermos5}.

\bibliography{entropy}

\def\cprime{$'$}
\begin{thebibliography}{10}

\bibitem{Vasseur}
F.~Berthelin, A.~E. Tzavaras, and A.~Vasseur.
\newblock From discrete velocity {B}oltzmann equations to gas dynamics before
  shocks.
\newblock {\em To appear in J. Stat. Phys.}, 2009.

\bibitem{Vasseur3}
F.~Berthelin and A.~Vasseur.
\newblock From kinetic equations to multidimensional isentropic gas dynamics
  before shocks.
\newblock {\em SIAM J. Math. Anal.}, 36(6):1807--1835 (electronic), 2005.

\bibitem{Brenier1}
Y.~Brenier.
\newblock Convergence of the {V}lasov-{P}oisson system to the incompressible
  {E}uler equations.
\newblock {\em Comm. Partial Differential Equations}, 25(3-4):737--754, 2000.

\bibitem{Brenier2}
Y.~Brenier, R.~Natalini, and M.~Puel.
\newblock On a relaxation approximation of the incompressible {N}avier-{S}tokes
  equations.
\newblock {\em Proc. Amer. Math. Soc.}, 132(4):1021--1028 (electronic), 2004.

\bibitem{Bressan}
A.~Bressan, T.-P. Liu, and T.~Yang.
\newblock {$L\sp 1$} stability estimates for {$n\times n$} conservation laws.
\newblock {\em Arch. Ration. Mech. Anal.}, 149(1):1--22, 1999.

\bibitem{Bressan2}
A.~Bressan and W.~Shen.
\newblock Uniqueness for discontinuous {ODE} and conservation laws.
\newblock {\em Nonlinear Anal.}, 34(5):637--652, 1998.

\bibitem{Chen1}
G.-Q. Chen, H.~Frid, and Y.~Li.
\newblock Uniqueness and stability of {R}iemann solutions with large
  oscillation in gas dynamics.
\newblock {\em Comm. Math. Phys.}, 228(2):201--217, 2002.

\bibitem{Colombo}
R.~M. Colombo and A.~Marson.
\newblock A {H}\"older continuous {ODE} related to traffic flow.
\newblock {\em Proc. Roy. Soc. Edinburgh Sect. A}, 133(4):759--772, 2003.

\bibitem{Dafermos1}
C.~M. Dafermos.
\newblock Generalized characteristics and the structure of solutions of
  hyperbolic conservation laws.
\newblock {\em Indiana Univ. Math. J.}, 26(6):1097--1119, 1977.

\bibitem{Dafermos3}
C.~M. Dafermos.
\newblock The second law of thermodynamics and stability.
\newblock {\em Arch. Rational Mech. Anal.}, 70(2):167--179, 1979.

\bibitem{Dafermos5}
C.~M. Dafermos.
\newblock Generalized characteristics in hyperbolic systems of conservation
  laws.
\newblock {\em Arch. Rational Mech. Anal.}, 107(2):127--155, 1989.

\bibitem{Dafermos4}
C.~M. Dafermos.
\newblock Entropy and the stability of classical solutions of hyperbolic
  systems of conservation laws.
\newblock In {\em Recent mathematical methods in nonlinear wave propagation
  ({M}ontecatini {T}erme, 1994)}, volume 1640 of {\em Lecture Notes in Math.},
  pages 48--69. Springer, Berlin, 1996.

\bibitem{Handbook}
C.~M. Dafermos and M.~Pokorny, editors.
\newblock {\em Handbook of differential equations: evolutionary equations. Vol.
  IV}.
\newblock Elsevier/North-Holland, Amsterdam, 2008.

\bibitem{DiPerna}
R.~J. DiPerna.
\newblock Uniqueness of solutions to hyperbolic conservation laws.
\newblock {\em Indiana Univ. Math. J.}, 28(1):137--188, 1979.

\bibitem{Filippov}
A.~F. Filippov.
\newblock {\em Differential equations with discontinuous righthand sides},
  volume~18 of {\em Mathematics and its Applications (Soviet Series)}.
\newblock Kluwer Academic Publishers Group, Dordrecht, 1988.
\newblock Translated from the Russian.

\bibitem{Serre}
H.~Freist{\"u}hler and D.~Serre.
\newblock {$L\sp 1$} stability of shock waves in scalar viscous conservation
  laws.
\newblock {\em Comm. Pure Appl. Math.}, 51(3):291--301, 1998.

\bibitem{Oleinik}
A.~M. Il{\cprime}in and O.~A. Ole{\u\i}nik.
\newblock Behavior of solutions of the {C}auchy problem for certain quasilinear
  equations for unbounded increase of the time.
\newblock {\em Dokl. Akad. Nauk SSSR}, 120:25--28, 1958.

\bibitem{Jones}
C.~K. R.~T. Jones, R.~Gardner, and T.~Kapitula.
\newblock Stability of travelling waves for nonconvex scalar viscous
  conservation laws.
\newblock {\em Comm. Pure Appl. Math.}, 46(4):505--526, 1993.

\bibitem{Otto}
B.~Jourdain, C.~Le~Bris, T.~Leli{\`e}vre, and F.~Otto.
\newblock Long-time asymptotics of a multiscale model for polymeric fluid
  flows.
\newblock {\em Arch. Ration. Mech. Anal.}, 181(1):97--148, 2006.

\bibitem{Kruzkov}
S.~N. Kru{\v{z}}kov.
\newblock First order quasilinear equations with several independent variables.
\newblock {\em Mat. Sb. (N.S.)}, 81 (123):228--255, 1970.

\bibitem{Liu1}
T.-P. Liu.
\newblock Nonlinear stability of shock waves for viscous conservation laws.
\newblock {\em Mem. Amer. Math. Soc.}, 56(328):v+108, 1985.

\bibitem{Liu2}
T.-P. Liu.
\newblock Pointwise convergence to shock waves for viscous conservation laws.
\newblock {\em Comm. Pure Appl. Math.}, 50(11):1113--1182, 1997.

\bibitem{Marson}
A.~Marson.
\newblock Nonconvex conservation laws and ordinary differential equations.
\newblock {\em J. London Math. Soc. (2)}, 69(2):428--440, 2004.

\bibitem{Vasseur2}
A.~Mellet and A.~Vasseur.
\newblock Asymptotic analysis for a {V}lasov-{F}okker-{P}lanck/compressible
  {N}avier-{S}tokes system of equations.
\newblock {\em Comm. Math. Phys.}, 281(3):573--596, 2008.

\bibitem{Temple}
B.~Temple.
\newblock No {$L\sb 1$}-contractive metrics for systems of conservation laws.
\newblock {\em Trans. Amer. Math. Soc.}, 288(2):471--480, 1985.

\bibitem{Vasseur4}
A.~Vasseur.
\newblock Strong traces for solutions of multidimensional scalar conservation
  laws.
\newblock {\em Arch. Ration. Mech. Anal.}, 160(3):181--193, 2001.

\bibitem{Yau}
H.-T. Yau.
\newblock Relative entropy and hydrodynamics of {G}inzburg-{L}andau models.
\newblock {\em Lett. Math. Phys.}, 22(1):63--80, 1991.

\end{thebibliography}

\end{document}